\documentclass[11pt,twoside, leqno]{article}

\usepackage{amsfonts}
\usepackage{mathrsfs}

\usepackage{amssymb}
\usepackage{amsmath}
\usepackage{amsthm}
\usepackage{color}

\allowdisplaybreaks

\def\rd{{\mathbb R}^d}

\def\zz{{\mathbb Z}}
\def\cc{{\mathbb C}}

\def\nn{{\mathbb N}}

\def\cx{{\mathcal X}}

\def\fz{\infty}
\def\az{\alpha}
\def\supp{{\mathop\mathrm{\,supp\,}}}
\def\dist{{\mathop\mathrm{\,dist\,}}}
\def\loc{{\mathop\mathrm{\,loc\,}}}

\def\lz{\lambda}
\def\dz{\delta}

\def\ez{\varepsilon}

\def\bz{\beta}
\def\fai{\varphi}
\def\gz{{\gamma}}
\def\oz{{\omega}}

\def\vz{\varphi}

\def\wz{\widetilde}

\def\hs{\hspace{0.3cm}}

\def\ls{\lesssim}

\def\gfz{\genfrac{}{}{0pt}{}}

\def\lo{{L^1(\mu)}}
\def\wlo{{L^{1,\,\fz}(\mu)}}
\def\lt{{L^2(\mu)}}
\def\lin{{L^\fz(\mu)}}

\def\lp{{L^p(\mu)}}

\def\rbmo{{\mathop\mathrm{RBMO}\,(\mu)}}

\newcommand{\n}{\nonumber}

\def\dint{\displaystyle\int}

\def\dfrac{\displaystyle\frac}
\def\dsup{\displaystyle\sup}

\def\r{\right}
\def\lf{\left}

\newtheorem{thm}{Theorem}[section]
\newtheorem{lem}[thm]{Lemma}

\newtheorem{rem}[thm]{Remark}

\newtheorem{defn}[thm]{Definition}

\numberwithin{equation}{section}

\begin{document}

\arraycolsep=1pt

\title{\Large\bf Boundedness of Maximal Calder\'on-Zygmund
Operators on Non-homogeneous Metric Measure Spaces\footnotetext{\hspace{-0.35cm}
\emph{2010 Mathematics Subject Classification}. {Primary 42B20; Secondary 42B25, 43A99.}
\endgraf\emph{Key words and phrases.} upper doubling measure,
geometrically doubling condition, maximal Calder\'on-Zygmund operator.
\endgraf
This project is  supported by the National Natural Science Foundation
of China (Grant No. 11171027) and the third (corresponding) author
is also supported by the Specialized Research Fund for the Doctoral Program of Higher Education
of China (Grant No. 20120003110003).}}
\author{Suile Liu,\ Yan Meng\ and\ Dachun Yang\footnote{Corresponding author}}
\date{ }
\maketitle

\vspace{-0.8cm}

\begin{center}
\begin{minipage}{13.5cm}\small
{\noindent{\bf Abstract.}  {\small Let $(\cx,\,d,\,\mu)$ be a metric measure
space and satisfy the so-called upper doubling condition and the
geometrically doubling condition. In this paper, the authors show that for the maximal
Calder\'on-Zygmund operator associated with a singular integral
whose kernel satisfies the standard size condition and the H\"ormander
condition, its $L^p(\mu)$ boundedness with $p\in(1,\infty)$ is
equivalent to its boundedness from $L^1(\mu)$ into
$L^{1,\infty}(\mu)$. Moreover, applying this, together with a new Cotlar type inequality, the
authors show that if the Calder\'on-Zygmund operator
$T$ is bounded on $L^2(\mu)$,
then the corresponding maximal Calder\'on-Zygmund is bounded on $L^p(\mu)$
for all $p\in(1,\infty)$, and bounded from $L^1(\mu)$ into $L^{1,\infty}(\mu)$.
These results essentially improve the existing results.}}
\end{minipage}
\end{center}

\section{Introduction}\label{sec1}

\hskip\parindent
It is well known that the Calder\'on-Zygmund theory is one of the core research areas
in harmonic analysis. It is intimately connected with partial differential equations,
operator theory, several complex variables and other fields.
During the development of the Calder\'on-Zygmund theory,
the only thing that has remained unchanged until
recently was the doubling property of the underlying measure.
We recall that the measure $\mu$ is said to satisfy the \emph{doubling condition} if there
exists a positive constant $C_\mu$ such that, for all $x\in\cx$ and $r\in(0,\,\infty)$,
\begin{equation}\label{eq1.1}
\mu(B(x,\,2r))\le C_\mu \mu(B(x,\,r)),
\end{equation}
where $(\cx,d)$ is some metric space endowed with
a nonnegative Borel measure $\mu$ and $B(x,\,r):=\{y\in\cx:\,\, d(y, x)< r\}$.
However, some research now indicates that
the doubling condition \eqref{eq1.1}
is superfluous for most results of the Calder\'on-Zygmund theory.
Let $\mu$ be a non-negative Radon measure on $\rd$, which is assumed to satisfy only
some polynomial growth condition, namely,
there exist positive constants $C_0$ and $n\in(0,\,d]$
such that, for all $x\in\rd$ and $r\in(0,\,\fz)$,
\begin{equation}\label{eq1.2}
\mu(B(x,\,r))\le C_0r^n,
\end{equation}
where $B(x,\,r):=\{y\in\rd:\ |y-x|<r\}$.
Obviously, such a measure is not necessary to satisfy the doubling condition
\eqref{eq1.1}. Under the assumption \eqref{eq1.2} on the measure $\mu$,
many results on the classical Calder\'on-Zygmund theory
have been proved to still hold; see, for example,
\cite{cmy,hmy,hyy09,ntv2,ntv03,t01,t01b,t03,mm,yyh13}
and some references therein. The motivation for developing the analysis on such $\rd$ can
be found in \cite{t03} and \cite{v}.
We only point out that the analysis in a such setting plays
an essential role in solving the long-standing open Painlev\'e's problem by Tolsa in \cite{t03}.

Notice that $\rd$ with the underlying measure as in \eqref{eq1.2} can not be encompassed
in the framework of spaces of homogeneous type in the sense of Coifman and Weiss \cite{cw71},
and vice verse. Recall that a metric space $(\cx,\,d)$ equipped with a nonnegative measure $\mu$ is
called a \emph{space of homogeneous type} if $(\cx,\,d,\,\mu)$
satisfies the doubling condition \eqref{eq1.1}.
The Calder\'on-Zygmund theory on $\rd$ with a measure $\mu$ satisfying \eqref{eq1.2}
is not in all respects a generalization
of the corresponding theory on spaces of homogeneous type
since the condition \eqref{eq1.2} on the measure is not more general than \eqref{eq1.1}.

Recently, Hyt\"onen \cite{h10} introduced a new class of metric measure
spaces satisfying the so-called upper doubling and the
geometrically doubling conditions (see also Definitions \ref{defn1.1} and
\ref{defn1.2} below).
This new class of metric measure spaces include both the spaces of homogeneous type and
metric spaces with the measures satisfying \eqref{eq1.2} as special cases.
In this setting, Hyt\"onen \cite{h10} introduced
the regularized $\mathop\mathrm{BMO}$ space.
Hyt\"onen and Martikainen \cite{hm} further
established a version of $T(b)$ theorem for Calder\'on-Zygmund operators.
Recently, Lin and Yang \cite{ly} introduced the space $\mathop\mathrm{RBLO}\,(\mu)$
and applied this space to the boundedness of the maximal
Calder\'on-Zygmund operators. Moreover, Hyt\"onen, Da. Yang and Do. Yang \cite{hyy}
studied the atomic Hardy space $H^1(\mu)$,
and Hyt\"onen, Liu, Da. Yang and Do. Yang \cite{hlyy}
established some equivalent characterizations for the boundedness of the Calder\'on-Zygmund
operators. Some of results in \cite{hlyy,hyy} were also independently obtained by
Bui and Duong \cite{ad} via different approaches.
Very recently, Hu, Meng and Yang \cite{hmy12-2} established
a new characterization of the space $\rbmo$ and proved
that the $L^p(\mu)$-boundedness with $p\in(1,\,\infty)$ of the
Calder\'on-Zygmund operator is equivalent
to its various endpoint estimates.
Some weighted norm inequalities for the multilinear
Calder\'on-Zygmund operators, via some weighted estimates involving the John-Str\"omberg maximal
operators and the John-Str\"omberg sharp maximal operators,
were also presented by Hu, Meng and Yang in \cite{hmy12-1}.
Fu, Yang and Yuan \cite{fyy12} proved that the multilinear commutators of Calder\'on-Zygmund operators
with $\rbmo$ functions are bounded on Orlicz spaces, especially, on $L^p(\mu)$ with $p\in(1,\,\infty)$,
and established the weak type endpoint estimate for the multilinear
commutators of Calder\'oon-Zygmund operators with Orlicz type functions in
$\textrm{Osc}_{\exp\,L^r} (\mu)$ for $r \in[1,\,\infty)$. More developments
of the harmonic analysis over this setting
are summarized in the monograph \cite{yyh13}.

The goal of this paper is two folds. One is to prove that, on the upper and
geometrically doubling metric measure spaces, for the maximal Calder\'on-Zygmund operator
whose kernel satisfies the standard size condition and the H\"ormander
condition, its $\lp$ boundedness with $p\in(1,\,\fz)$ is equivalent
to its boundedness from $\lo$ into $\wlo$.
Based on this equivalence and a Cotlar type inequality established in this paper,
we then establish the boundedness of the maximal Calder\'on-Zygmund operator
on $\lp$ for all $p\in(1,\,\fz)$ and  from $\lo$ into $\wlo$ under the assumption that
the Calder\'on-Zygmund operator is bounded on $L^2(\mu)$.

To state the results of this paper, we first recall some necessary
notions and notation. We start with the notion
of the upper doubling and geometrically doubling metric
measure space introduced in \cite{h10}.

\begin{defn}\label{defn1.1}\rm
A metric measure space $(\cx,\,d,\,\mu)$ is said to be \emph{upper
doubling} if $\mu$ is a Borel measure on $\cx$ and there exist a
\emph{dominating function} $\lz:\,\,\cx\times(0,\,\fz)\to (0,\,\fz)$
and a positive constant $C_\lz$, depending on $\lz$,
such that for each $x\in\cx$,
$r\to\lz(x,\,r)$ is non-decreasing and, for all $x\in\cx$ and
$r\in(0,\,\fz)$,
\begin{equation}\label{eq1.3}
\mu(B(x,\,r))\le\lz(x,\,r)\le C_\lz\lz(x,\,r/2),
\end{equation}
here  and in what follows, $B(x,\,r):=\{y\in \cx:\ d(y,\,x)<r\}$.
\end{defn}

\begin{rem}\label{rem1.1}\rm

(i) Obviously, a space of homogeneous type, $(\cx,\,d,\,\mu)$,
is a special case of upper doubling spaces if taking  $\lz(x,\,r):= \mu(B(x,\,r))$.
On the other hand,  $(\rd,\,|\cdot|,\,\mu)$ with $\mu$
satisfying the  polynomial growth condition \eqref{eq1.2} is also
an upper doubling measure space by taking $\lz(x,\,r):= C_0r^n$,
where $C_0$ is as in \eqref{eq1.2}.

(ii) Let $(\cx,\,d,\,\mu)$ be an upper doubling space and $\lz$ a
dominating function on $\cx\times(0, \fz)$ as in Definition
\ref{defn1.1}. It was proved in \cite{hyy} that there exists another dominating function
$\wz\lz$ related to $\lz$
such that $\wz\lz\le\lz$, $C_{\wz\lz}\le C_\lz$ and,
for all $x$, $y\in\cx$ with $d(x,\,y)\le r$,
\begin{equation}\label{eq1.4}
\wz\lz(x,\,r)\le C_{\wz\lz}\wz\lz(y,\,r).
\end{equation}
Thus, in the below of this paper, we \emph{always assume that $\lz$ satisfies \eqref{eq1.4}}.
\end{rem}

\begin{defn}\label{defn1.2}\rm
A metric space $(\cx,\,d)$ is said to be \emph{geometrically doubling}
if there exists some $N_0\in\nn:=\{1,\,2,\,\cdots\}$ such that for
any ball $B(x,\,r)\subset\cx$, there exists a finite ball covering
$\{B(x_i,\,r/2)\}_i$ of $B(x,\,r)$ such that the \emph{cardinality} of this
covering is at most $N_0$.
\end{defn}

\begin{rem}\label{rem1.2}\rm
Let $(\cx,\,d)$ be a metric space. In \cite{h10},
Hyt\"onen proved that the following statements are mutually
equivalent:
\vspace{-0.3cm}
\begin{itemize}
  \item [(i)] $(\cx,\,d)$ is geometrically doubling.
  \vspace{-0.3cm}
  \item [(ii)] For any $\ez\in (0,\,1)$ and any ball
  $B(x,\,r)\subset\cx$, there exists a finite ball
  covering $\{B(x_i,\,\ez r)\}_i$ of $B(x,\,r)$ such that
  the cardinality of this covering is at most $N_0\ez^{-n}$, here
  and in what follows, $N_0$ is as in Definition \ref{defn1.2} and
  $n:=\log_2N_0$.
  \vspace{-0.3cm}
  \item [(iii)] For every $\ez\in (0,\,1)$, any ball
  $B(x, r)\subset\cx$ contains at most
  $N_0\ez^{-n}$ centers $\{x_i\}_i$ of disjoint
  balls $\{B(x_i,\,\ez r)\}_i$.
  \vspace{-0.3cm}
  \item [(iv)]  There exists $M\in\nn$ such that any ball
  $B(x,\,r)\subset\cx$ contains at most
  $M$ centers $\{x_i\}_i$ of disjoint
  balls $\{B(x_i,\,r/4)\}_{i=1}^M$.
\end{itemize}
\vspace{-0.3cm}
\end{rem}

We now recall the notions of Calder\'{o}n-Zygmund operators and the corresponding
maximal Calder\'{o}n-Zygmund operators in the present context.

Let $\triangle:=\{(x, x):\,\, x\in\cx\}$ and $K$ be a $\mu$-locally integrable function mapping
from $(\cx\times\cx)\setminus\triangle$ to $\cc$,
which satisfies the \emph{size condition} that  there exists a positive constant $C$
such that, for all $x$, $y\in\cx$ with $x\not=y$,
\begin{equation}\label{eq1.5}
|K(x,\,y)|\le C\dfrac1{\lz(x, d(x,\,y))},
\end{equation}
and the \emph{H\"ormander condition} that there exists a positive constant $C$
such that, for all $x,\,\wz x\in\cx$ with $x\not=\wz x$,
\begin{equation}\label{eq1.6}
\int_{d(x,\,y)\ge 2d(x,\,\wz x)}\lf[|K(x,\,y)-K(\wz x,\,y)|
+|K(y,\,x)-K(y,\,\wz x)|\r]\,d\mu(y)\le C.
\end{equation}
A linear operator $T$ is called a \emph{Calder\'on-Zygmund operator}
with the kernel $K$ satisfying \eqref{eq1.5} and \eqref{eq1.6} if,
for all $f\in L^\fz_b(\mu)$, the \emph{space of bounded functions
with bounded supports}, and $x\notin\supp f$,
\begin{equation}\label{eq1.7}
T(f)(x):= \dint_\cx K(x,\,y)f(y)\,d\mu(y).
\end{equation}
Let $\ez\in(0,\,\infty)$. The \emph{truncated operator} $T_\ez$ is defined by setting,
for all $f\in L^\fz_b(\mu)$ and $x\in\cx$,
$$
T_\ez(f)(x):=\int_{d(x,\,y)>\ez}K(x,\,y)f(y)\,d\mu(y).
$$
Moreover, the \emph{maximal Calder\'{o}n-Zygmund operator} $T_\ast$,
associated with $\{T_\ez\}_{\ez>0}$, is defined
by setting, for all $f\in L^\fz_b(\mu)$ and $x\in\cx$,
\begin{equation}\label{eq1.9}
T_\ast(f)(x):=\sup_{\ez\in(0,\,\infty)}|T_\ez(f)(x)|.
\end{equation}

The main results of this paper are as follows.

\begin{thm}\label{thm1.1}
Let $K$ be a $\mu$-locally integrable function mapping from
$(\cx\times\cx)\setminus\triangle$ to $\cc$,
which satisfies \eqref{eq1.5} and \eqref{eq1.6}, and
$T$ as in \eqref{eq1.7}.
Suppose that $T$ is bounded on $\lt$. Then
\vspace{-0.3cm}
\begin{itemize}
\item [\rm(i)] the corresponding maximal
Calder\'on-Zygmund operator $T_\ast$ as in \eqref{eq1.9}
is bounded on $\lp$ for all $p\in(1,\,\fz)$;
\vspace{-0.3cm}
\item [\rm(ii)] $T_\ast$ is bounded from $\lo$ into $\wlo$.
\end{itemize}
\vspace{-0.3cm}
\end{thm}

Indeed, Bui and Duong \cite[Remark 6.7]{ad} have obtained the boundedness
on Lebesgue spaces $L^p(\mu)$ for $p\in(1,\,\infty)$ of $T_\ast$
with the kernel satisfying \eqref{eq1.5} and the following \emph{stronger regularity condition}, that is,
there exist positive constants $C$ and $\tau\in(0,1]$
such that, for all $x,\,\wz{x},\,y\in\cx$ with $d(x,\,y)\ge 2d(x,\,\wz x)$,
\begin{equation}\label{eq1.11}
|K(x,\,y)-K(\wz x,\,y)|+|K(y,\,x)-K(y,\,\wz x)|
\le C\frac{[d(x,\,\wz x)]^\tau}{[d(x,\,y)]^\tau\lz(x,\,d(x,y))}.
\end{equation}
A new example of operators with the kernels satisfying \eqref{eq1.5} and \eqref{eq1.11}
is the so-called Bergman-type operator appearing in \cite{vw09};
see also \cite{hm} for an explanation.
Theorem \ref{thm1.1} essentially improves \cite[Remark 6.7]{ad}, since the kernel
in Theorem \ref{thm1.1}  is assumed to satisfy the
weaker regularity condition \eqref{eq1.6}.

We remark that if $\cx$ is separable and the kernel $K$ satisfies
\eqref{eq1.5} and \eqref{eq1.11}, Theorem \ref{thm1.1}
has been established in \cite[Corollary 1.7]{hlyy}. Thus,
Theorem \ref{thm1.1} also essentially improves
\cite[Corollary 1.7]{hlyy}.

To prove Theorem \ref{thm1.1}, we establish the following
equivalent characterization of the $\lp$-boundedness with $p\in(1,\,\fz)$
and the weak type (1,1) estimate for the operator $T_\ast$.

\begin{thm}\label{thm1.2}
Let $K$ be a $\mu$-locally integrable function mapping from
$(\cx\times\cx)\setminus\triangle$ to $\cc$,
which satisfies \eqref{eq1.5} and \eqref{eq1.6}, and
$T_\ast$ be as in \eqref{eq1.9}.
Then the following three statements are equivalent:
\vspace{-0.3cm}
\begin{itemize}
\item [\rm(i)] $T_\ast$ is bounded on $L^{p_0}(\mu)$ for some $p_0\in(1,\,\fz)$;
\vspace{-0.3cm}
\item [\rm(ii)] $T_\ast$ is bounded from $\lo$ into $\wlo$;
\vspace{-0.3cm}
\item [\rm(iii)] $T_\ast$ is bounded on $\lp$ for all $p\in(1,\,\fz)$.
\end{itemize}
\vspace{-0.3cm}
\end{thm}

This paper is organized as follows.  In Section \ref{sec2},
we give the proof of  Theorem \ref{thm1.2}.
As an application of Theorem \ref{thm1.2}, in Section \ref{sec3},
we prove Theorem \ref{thm1.1}.

Indeed, to show Theorem \ref{thm1.1}, we first establish a
Cotlar type inequality in Theorem \ref{thm3.1} below
and then, using this inequality, show that the maximal
Calder\'on-Zygmund operator $T_\ast$ in Theorem \ref{thm1.1}
is bounded from $L^1(\mu)$ into $L^{1,\,\infty}(\mu)$.
Furthermore, applying Theorem \ref{thm1.2},
we then obtain the boundedness of $\lp$ with $p\in(1,\,\fz)$ for $T_\ast$
and hence complete the proof of Theorem \ref{thm1.1}.

We point out that, in the proofs of Theorems \ref{thm1.1}
and \ref{thm1.2}, we do borrow some ideas from the proofs
of \cite[Theorems 1.1 and 1.2]{hmy06} which are for
the case that $\cx=\rd$ and $\mu$ being a Radon
measure as in \eqref{eq1.2}. However, there exists
a gap in the proof of \cite[Theorem 1.2]{hmy06},
which is sealed in the proof of Theorem \ref{thm1.2} below. Precisely, in the proof
that the weak type $(1,\,1)$ estimate implies the $L^p(\mu)$ boundedness,
with $p\in(1,\,\infty)$, of the maximal Calder\'on-Zygmund operator
$T_\ast$ in \cite[Theorem 1.2]{hmy06} (see also (ii) $\Longrightarrow$ (iii)
in the proof of Theorem \ref{thm1.2} below), the sharp
maximal function estimate of the maximal Calder\'on-Zygmund operator $T_\ast$ was used.
Based on this estimate and the Marcinkiewicz interpolation theorem,
the $L^p(\mu)$ boundedness, with $p\in(1,\,\infty)$,
of the operator, $M^\sharp_r\circ T_\ast$, compounded by the sharp maximal
function $M^\sharp_r$ (see \eqref{2.1x} below for the definition)
and the maximal Calder\'on-Zygmund
operator $T_\ast$ was then concluded in the proof of \cite[Theorem 1.2]{hmy06}.
However, it is not clear whether the operator $M^\sharp_r\circ T_\ast$
is quasi-linear and hence applying the Marcinkiewicz interpolation theorem
directly to the operator $M^\sharp_r\circ T_\ast$ might be problematic.
To avoid this, in the below proof of (ii) $\Longrightarrow$ (iii) of Theorem \ref{thm1.2},
we borrow some new ideas from the proof of \cite[Lemma 3]{hlw}.
Without aid of the quasi-linear property of this compound operator,
we show the operator $M^\sharp_r\circ T_\ast$ is bounded from $L^p(\mu)$
into $L^{p,\,\infty}(\mu)$ with $p\in(1,\,\infty)$,
which implies that the maximal Calder\'on-Zygmund operator $T_\ast$ is also bounded from
$L^p(\mu)$ into $L^{p,\,\infty}(\mu)$ with $p\in(1,\,\infty)$. Notice that
the maximal Calder\'on-Zygmund operator $T_\ast$ is sublinear. Then, by the Marcinkiewicz interpolation theorem,
we conclude that the maximal Calder\'on-Zygmund operator $T_\ast$ is bounded on  $L^p(\mu)$
with $p\in(1,\,\infty)$, which is the desired conclusion (iii) of Theorem \ref{thm1.2}.

We finally make some conventions on notation.
Throughout this paper, we always denote by $C$, $\wz C$, $c$
and $\wz c$ \emph{positive constants} which
are independent of the main parameters, but
they may vary from line to
line. Positive constants with subscripts,
such as $C_1$, do not change in different
occurrences. Furthermore, we use $C_\az$ to denote
a \emph{positive constant depending on the parameter $\az$}.
The \emph{symbol $Y\ls Z$} means that
there exists a positive constant $C$ such that $Y\le CZ$.
The \emph{symbol $A\sim B$} means that $A\ls B\ls A$.
For any ball $B\subset{\mathcal X}$, we denote its \emph{center} and \emph{radius},
respectively, by $x_B$ and $r_B$
and, moreover, for any $\rho\in (0,\,\fz)$, the \emph{ball} $B(x_B,\,\rho r_B)$ by $\rho B$.
Given any $q\in(1,\,\fz)$, let $q':= q/(q-1)$ denote its
\emph{conjugate index}. Also, for any subset $E\subset{\mathcal X}$,
$\chi_E$ denotes its \emph{characteristic function}.

\section{Proof of Theorem \ref{thm1.2}}\label{sec2}

\hskip\parindent This section is devoted to the proof of Theorem
\ref{thm1.2}. To this end, we first recall some useful notions.

The measure as in \eqref{eq1.3}
is not necessary to satisfy the doubling condition \eqref{eq1.1}.
However, there still exist a lot of doubling balls in the present context.
Given $\az$, $\bz\in(1,\,\fz)$, a ball $B\subset\cx$ is called
\emph{$(\az,\,\bz)$-doubling} if $\mu(\az B)\le\bz\mu(B)$.
It was proved in \cite[Lemmas 3.2 and 3.3]{h10} that if a metric
measure space $(\cx,\,d,\,\mu)$ is upper doubling and $\bz>C_\lz^{\log_2\az}=:\az^\nu$,
then for every ball $B\subset\cx$, there exists some $j\in\zz_+:=\nn\cup\{0\}$
such that $\az^j B$ is $(\az,\,\bz)$-doubling
and, on the other hand, if $(\cx,\,d)$ is geometrically doubling
and $\bz>\az^n$ with $n:= \log_2 N_0$, then for $\mu$-almost every
$x\in\cx$, there exist $(\az,\,\bz)$-doubling balls with arbitrarily small radiuses
of the form $B(x,\,\az^{-j}r)$ for some $j\in\nn$ and any preassigned $r\in(0,\,\fz)$.
Throughout this paper, for any $\az\in (1,\,\fz)$ and ball $B$, $\wz B^\az$ denotes
the \emph{smallest $(\az,\,\bz_\az)$-doubling ball of the form $\az^j B$
with $j\in\zz_+$}, where $\bz_\az>\max\lf\{\az^n,\,\az^\nu\r\}$.
If $\az=6$, we denote the ball $\wz B^\az$ simply by $\wz B$.

For all balls $B\subset S\subset\cx$, define
$$\dz(B,\,S):=\int_{(2S)\setminus B}
\frac1{\lz(x_B, d(x,\,x_B))}\,d\mu(x).$$
The coefficient $\dz(B,\,S)$ was introduced in \cite{h10}, which is analogous to the quantity
$K_{Q,\,R}$ introduced by Tolsa \cite{t01} (see also \cite{t01b,t4}).
For $\delta(B,\,S)$, we have the following useful
properties (see, for example, \cite[Lemmas 5.1 and 5.2]{h10} and \cite[Lemma 2.2]{hyy}).

\begin{lem}\label{lem2.1}
\begin{itemize}
\item [\rm(i)] For all balls $B\subset R\subset S$, $\dz(B,\,R)\le \dz(B,\,S)$.
\vspace{-0.3cm}
\item [\rm(ii)] For all $\rho\in[1,\,\fz)$, there exists a positive
constant $C_\rho$, depending on $\rho$, such that, for all balls
$B\subset S$ with $r_S\le \rho r_B$, $\dz(B,\,S)\le C_\rho$.
\vspace{-0.3cm}
\item [\rm(iii)] For all $\az\in(1,\,\fz)$,
there exists a positive constant $C_\az$, depending on $\az$,
such that, for all balls $B$, $\dz(B,\,\wz B^\az)\le C_\az$.
\vspace{-0.3cm}
\item [\rm(iv)] There exists a positive constant $c$ such that, for
all balls $B\subset R\subset S$, $\dz(B,\,S)\le \dz(B,\,R)+ c\dz(R,\,S)$.
In particular, if $B$ and $R$ are concentric, then $c=1$.
\vspace{-0.3cm}
\item [\rm(v)] There exists a positive constant $\wz c$ such that, for all
balls $B\subset R\subset S$, $\dz(R,\,S)\le\wz c[1+\dz(B,\,S)]$; moreover,
if $B$ and $R$ are concentric, then $\dz(R,\,S)\le \dz(B,\,S)$.
\end{itemize}
\end{lem}

To prove Theorem \ref{thm1.1}, we now recall the
following Calder\'on-Zygmund decomposition
from \cite[Theorem 6.3]{ad}. Let $\gz_0$ be a fixed positive constant
satisfying that $\gz_0>\max\{C_\lz^{3\log_26}, 6^{3n}\}$, where
$C_\lz$ is as in \eqref{eq1.3} and $n$ is
as in Remark \ref{rem1.2}{\rm(ii)}.

\begin{lem}\label{lem2.2}
Let $p\in[1,\,\fz)$, $f\in L^p(\mu)$ and $t\in(0,\,\fz)$
($t>\gz_0^{1/p}\|f\|_\lp/[\mu(\cx)]^{1/p}$ when $\mu(\cx)<\fz$). Then
\begin{itemize}
\vspace{-0.3cm}
\item[\rm(i)] there exists a family of finite overlapping balls, $\{6B_i\}_i$, such
that $\{B_i\}_i$ is pairwise disjoint,
\begin{equation}\label{eq2.2}
\frac{1}{\mu\lf(6^2B_i\r)}\int_{B_i}|f(x)|^p\,d\mu(x)>
\frac {t^p}{\gz_0}\
\textrm{for all}\ i,\end{equation}
$$
\frac1{\mu(6^2\eta B_i)}\int_{\eta B_i}|f(x)|^p\,d\mu(x)
\le\frac{t^p}{\gz_0}\ \textrm{for all}\ i\
\textrm{and}\ \eta\in(2,\,\infty),
$$
and
\begin{eqnarray}\label{eq2.4}
 |f(x)|\le t \,\,\textrm{for}\,\,\mu\textrm{-almost every}\ x\in\cx\setminus(\cup_i 6B_i);
\end{eqnarray}

\vspace{-0.3cm}
\item[\rm(ii)] for each $i$, let $S_i$ be the smallest
$(3\times 6^2, C_\lz^{\log_2(3\times6^2)+1})$-doubling ball of
the family $\{(3\times 6^2)^k B_i\}_{k\in\nn}$, and
$\omega_i:=\chi_{6B_i}/(\sum_k \chi_{6B_k})$.
Then there exists a family $\{\vz_i\}_i$ of
functions such that for each $i$, $\supp(\vz_i)\subset S_i$,
$\vz_i$ has a constant sign on $S_i$,
$$
\dint_\cx \vz_i(x)\,d\mu(x)=\dint_{6B_i} f(x)\omega_i(x)\,d\mu(x)
$$
and
\begin{equation}\label{eq2.6}
\sum_{i}|\vz_i(x)|\le \gz t \,\,\textrm{for}\,\,\mu\textrm{-almost every\,\,} x\in\cx,
\end{equation}
where $\gz$ is a positive constant depending only on
$(\cx, \mu)$ and there exists a positive constant $C$, independent of
$f$, $t$ and $i$, such that, if $p=1$,
\begin{equation}\label{eq2.7}
\|\vz_i\|_\lin\mu(S_i)\le C\dint_{\cx}|f(x)\omega_i(x)|\,d\mu(x)
\end{equation}
and, if $p\in (1,\,\fz)$,
$$
\lf\{\dint_{S_i}|\vz_i(x)|^p\,d\mu(x)\r\}^{1/p}
[\mu(S_i)]^{1/p'}\le \frac{C}{t^{p-1}}\dint_{\cx}|f(x)\omega_i(x)|^p\,d\mu(x).
$$
\end{itemize}
\end{lem}

We now recall some maximal functions in \cite{h10, ad} as follows.
The \emph{non-centered doubling maximal function} $N$ is defined by setting,
for all $f\in L^1_\loc(\mu)$ and $x\in\cx$,
\begin{eqnarray*}
N(f)(x):=\sup_{\gfz{B\ni x}{B\,(6,\,\bz_6)-{\rm doubling}}}\frac{1}{\mu(B)}\int_B|f(y)|\,d\mu(y).
\end{eqnarray*}
By the Lebesgue differential theorem, we see that, for $\mu$-almost every $x\in\cx$,
\begin{equation}\label{eq2.9x}
|f(x)|\le N(f)(x);
\end{equation}
see \cite[Corollary 3.6]{h10}. Let $\eta\in(1,\,\infty)$.
The \emph{Hardy-Littlewood maximal function}
and the \emph{sharp maximal function} are, respectively, defined by
setting, for all $f\in L^1_\loc(\mu)$ and $x\in\cx$,
\begin{equation*}
\begin{array}[t]{cl}
$$M_{(\eta)}(f)(x):=\dsup_{B\ni x}\frac{1}{\mu(\eta B)}
\int_B|f(y)|\,d\mu(y)$$
\end{array}
\end{equation*}
and
\begin{eqnarray*}
M^\sharp(f)(x):=&\dsup_{B\ni x}\frac{1}{\mu(5B)}
\int_B|f(y)-m_{\wz B}(f)|\,d\mu(y)
+\dsup_{\gfz{x\in B\subset S}{B,\,S\,(6,\,\bz_6)-{\rm doubling}}}
\frac{|m_B(f)-m_S(f)|}{1+\dz(B,\,S)},\nonumber
\end{eqnarray*}
where $m_B(f)$ denotes the \emph{mean of $f$ over $B$}, that is,
$m_B(f):= \frac1{\mu(B)}\int_Bf(x)\,d\mu(x)$.
Moreover, for all $r\in(0,\,\infty)$, the \emph{operators} $M_{r,\,(\eta)}$ and $M^\sharp_r$
are defined by setting,  for all $f\in L_\loc^r(\mu)$ and $x\in\cx$,
$$M_{r,\,(\eta)}(f)(x):=\sup_{B\ni x}\lf\{\frac{1}{\mu(\eta B)}
\int_B|f(y)|^r\,d\mu(y)\r\}^{\frac{1}{r}}$$
and
\begin{eqnarray}\label{2.1x}
M_r^\sharp(f)(x):=\{M^\sharp(|f|^r)(x)\}^{\frac{1}{r}},
\end{eqnarray}
respectively.
By \cite[Proposition 3.5]{h10} and \cite[Lemma 2.3]{hlyy},
we see that for each $\eta\in[5,\,\infty)$, $M_{(\eta)}$ is bounded
from $L^p(\mu)$ into itself with $p\in(1, \fz)$ and from $L^1(\mu)$ into $L^{1,\,\infty}(\mu)$.

\begin{lem}[{\cite[Lemma 3.3]{lyy}}]\label{lem2.1x}
For all $f\in L_\loc^1(\mu)$, with $\int_\cx f(x)\,d\mu(x)=0$
when $\mu(\cx)<\fz$, if $\min\{1,N(f)\}\in L^{p_0}(\mu)$ for some $p_0\in(1,\,\infty)$,
then for all $p\in[p_0,\,\infty)$, there exists a positive constant $C_p$,
depending on $p$ but independent of $f$, such that
$$
\|N(f)\|_{L^{p,\,\infty}(\mu)}\le C_p\|M^\sharp(f)\|_{L^{p,\,\infty}(\mu)}.
$$
\end{lem}

Based on the above lemmas, we now turn to the proof of Theorem \ref{thm1.2}.
\begin{proof}[Proof of Theorem \ref{thm1.2}]
We first show that (i) $\Longrightarrow$ (ii). To this end,
assume that $T_\ast$ is bounded on $L^{p_0}(\mu)$ for some $p_0\in(1,\,\infty)$.
Our goal is to show that there exists a positive constant $C$ such that, for all $f\in\lo$ and $t\in(0,\,\fz)$,
\begin{equation}\label{eq2.12}
\mu(\{x\in\cx:\ T_\ast(f)(x)>t\})\le\frac{C}{t}\int_\cx|f(x)|\,d\mu(x).
\end{equation}
Observe that if $\mu(\cx)<\fz$ and $t\le{\gz_0\|f\|_\lo}/{\mu(\cx)}$,
then the inequality \eqref{eq2.12} is trivial. Therefore, we may assume that
$t>{\gz_0\|f\|_\lo}/{\mu(\cx)}$ when $\mu(\cx)<\fz$.
For each fixed $t\in(0,\,\fz)$ and $f\in\lo$, applying Lemma
\ref{lem2.2} with $p=1$ to $f$ at level $t$ with the notation same as in Lemma \ref{lem2.2},
we obtain $f=g+h$, where
$$g:= f\chi_{\cx\setminus(\cup_j 6B_j)}
+\sum_j\varphi_j$$
and
$$h:= f-g=\sum_j[\oz_jf-\varphi_j]=:\sum_j h_j.$$
By \eqref{eq2.4}, \eqref{eq2.6} and \eqref{eq2.7}
in Lemma \ref{lem2.2}, it is easy to see that, for $\mu$-almost every $x\in\cx$,
\begin{equation}\label{eq2.14x}
|g(x)|\le|f(x)\chi_{\cx \setminus(\cup_j6B_j)}(x)|+\lf|\sum_j
\varphi_j(x)\r|\ls t
\end{equation}
and
\begin{eqnarray}\label{eq2.15x}
\|g\|_\lo&&\le\int_\cx|f(x)|\,d\mu(x)
+\int_\cx\lf|\sum_j\varphi_j(x)\r|\,d\mu(x)\\\nonumber
&&\le\|f\|_\lo+\sum_j\int_\cx|f(x)\oz_j(x)|\,d\mu(x)\ls\|f\|_\lo.
\end{eqnarray}

From $L^{p_0}(\mu)$ boundedness of $T_\ast$,
\eqref{eq2.14x} and \eqref{eq2.15x}, it follows that
\begin{align}\label{eq2.x1}
\mu(\{x\in\cx:\,|T_\ast(g)(x)|>t\})
&\ls\frac 1{t^{p_0}}\int_\cx
|T_\ast(g)(x)|^{p_0}\,d\mu(x)
\ls \frac 1{t^{p_0}}\|g\|_{L^{p_0}(\mu)}^{p_0}\\\n
&\ls\frac{1}{t^{p_0}}t^{p_0-1}\int_\cx|g(x)|\,d\mu(x)
\sim\frac{1}{t}\|g\|_\lo\ls
\frac{1}{t}\|f\|_\lo.
\end{align}

Moreover, by \eqref{eq2.2} with $p=1$, we conclude that
\begin{equation}\label{eq2.x2}
\mu\lf(\bigcup_j6^2B_j\r)\ls
t^{-1}\sum_j\int_{B_j}|f(x)|\,d\mu(x)
\ls t^{-1}\int_\cx|f(x)|\,d\mu(x).
\end{equation}
Hence, by \eqref{eq2.x1} and \eqref{eq2.x2}, together
with $T_\ast f\le T_\ast g+T_\ast h$, we see that the proof of \eqref{eq2.12}
is reduced to proving that
\begin{equation}\label{eq2.13}
\mu\lf(\lf\{x\in\cx\setminus
\lf(\bigcup_j6^2B_j\r):\ |T_\ast(h)(x)|>t\r\}\r)\ls\frac{1}{t}\|f\|_\lo.
\end{equation}

For each fixed $x\in\cx\setminus(\cup_j6^2B_j)$, write
\begin{align}\label{eq2.x3}
T_\ast(h)(x)&=\sup_{\ez\in(0,\,\infty)}\lf|T_\ez\lf(\sum_j h_j\r)(x)\r|\\\n
&\le\sup_{\ez\in(0,\,\infty)}\lf|\sum_j T_\ez(h_j)(x)
\chi_{6S_j\setminus6^2B_j}(x)\r|
+\sup_{\ez\in(0,\,\infty)}\lf|\sum_j T_\ez(h_j)(x)
\chi_{\cx\setminus6S_j}(x)\r|\\\n
&=:{\rm A}_1(x)+{\rm A}_2(x).
\end{align}

We first estimate ${\rm A}_1(x)$. To this end,
for all $x\in\cx\setminus(\cup_j6^2B_j)$, by \eqref{eq1.4} and \eqref{eq1.5},
we further write
\begin{eqnarray}\label{eq2.x4}
{\rm A}_1(x)&\le&\sup_{\ez\in(0,\,\infty)}\lf|\sum_jT_\ez(f\oz_j)(x)
\chi_{6S_j\setminus6^2B_j}(x)\r|\\\n
&\hs&+\sup_{\ez\in(0,\,\infty)}\lf|\sum_jT_\ez(\varphi_j)(x)
\chi_{6S_j\setminus6^2B_j}(x)\r|\\\n
&\le&\sum_j\int_{B_j}\frac{|f(y)|}{\lz(x,\,d(x,\,y))}
\,d\mu(y)\chi_{6S_j\setminus6^2B_j}(x)
+\sum_j T_\ast(\varphi_j)(x)\chi_{6S_j\setminus6^2B_j}(x)\\\n
&\ls&\sum_j\frac{\chi_{6S_j\setminus6^2B_j}(x)}
{\lz(x,\,d(x,\,x_{B_j}))}\int_{B_j}|f(y)|\,d\mu(y)
+\sum_j T_\ast(\varphi_j)(x)\chi_{6S_j}(x)\\\n
&=:&{\rm A}_{1,\,1}(x)+{\rm A}_{1,\,2}(x),
\end{eqnarray}
where the last inequality is justified by the fact that
$d(x,\,y)\sim d(x,x_{B_j})$ for all $x\in6S_j\setminus6^2B_j$ and $y\in B_j$.

We first estimate ${\rm A}_{1,\,1}(x)$.
Denote by $N_{6B_j,\,6S_j}$ the first positive integer $k$ such that
$6^kB_j\supset 6S_j$, and write $N_{6B_j,\,6S_j}$ simply by $N_j$.
Recalling that $S_j$ is the smallest $(3\times 6^2,\,C_\lz^{\log_2(3\times6^2)+1})$-doubling ball of
the family $\{(3\times 6^2)^k B_j\}_{k\in\nn}$,
then from \eqref{eq1.4} and Lemma \ref{lem2.1}(iii),
we deduce that
\begin{eqnarray*}
\int_{6S_j\setminus 6B_j}\frac{1}{\lz(x,\,d(x,\,x_{B_j}))}\,d\mu(x)
&\le&\sum_{k=1}^{N_j-1}\int_{(6^{k+1}B_j)\setminus(6^kB_j)}
\frac{1}{\lz(x,\,d(x,\,x_{B_j}))}\,d\mu(x)\\
&\ls&\sum_{k=1}^{N_j-1}\int_{(6^{k+1}B_j)\setminus(6^kB_j)}
\frac{1}{\lz(x_{B_j},\,d(x,\,x_{B_j}))}\,d\mu(x)\\
&\ls&\delta(6B_j,\,6S_j)+1\ls1,
\end{eqnarray*}
where in the last-to-second inequality, we used the fact that
\begin{eqnarray*}
\int_{6^{N_j}B_j\setminus6^{N_j-1}B_j}
\frac{1}{\lz(x_{B_j},\,d(x,\,x_{B_j}))}\,d\mu(x)&&\le \int_{6^{N_j}B_j\setminus6^{N_j-1}B_j}
\frac{1}{\lz(x_{B_j},\,6^{N_j-1}r_{B_j})}\,d\mu(x)\\
&&\le \frac{\mu(6^{N_j}B_j)}{\lz(x_{B_j},\,6^{N_j-1}r_{B_j})}\ls1.
\end{eqnarray*}
This implies that
$$
\mu\lf(\lf\{x\in\cx\setminus\lf(\bigcup_j6^2B_j\r):\ {\rm A}_{1,\,1}(x)>t\r\}\r)\ls
t^{-1}\sum_j\int_{B_j}|f(y)| \,d\mu(y)\ls t^{-1}\|f\|_\lo.
$$

To estimate ${\rm A}_{1,\,2}(x)$, by the H\"older inequality,
$L^{p_0}(\mu)$ boundedness of $T_{\ast}$, the support condition of $\fai_j$
and the fact that $S_j$ is a $(3\times 6^2,\,C_\lz^{\log_2(3\times6^2)+1})$-doubling ball,
together with \eqref{eq2.7},
we conclude that
\begin{eqnarray*}
\mu\lf(\lf\{x\in\cx\setminus\lf(\bigcup_j6^2B_j\r):
\ {\rm A}_{1,\,2}(x)>t\r\}\r)&\le& t^{-1}\sum_j \|T_\ast\fai_j\|_{L^{p_0}(\mu)}
[\mu(6S_j)] ^\frac{1}{p'_0}\\
&\ls& t^{-1}\sum_j\|\fai_j\|_{L^{p_0}(\mu)}[\mu(6S_j)]^\frac{1}{p'_0}\\
&\ls& t^{-1}\sum_j\|\fai_j\|_{L^\infty(\mu)}\mu(S_j)
\ls t^{-1}\|f\|_\lo,
\end{eqnarray*}
which, along with the estimate for ${\rm A}_{1,\,1}(x)$, implies that
$$
\mu\lf(\lf\{x\in\cx\setminus\lf(\bigcup_j6^2B_j\r):\ {\rm A}_{1}(x)>t\r\}\r)\ls
t^{-1}\sum_j\int_{B_j}|f(x)| \,d\mu(x)\ls t^{-1}\|f\|_\lo.
$$

To estimate ${\rm A}_2(x)$,
we consider the following three cases:

{\rm Case (i)} $\ez\in(0,\,\dist(x,\,S_j))$, where $\dist(x,\,S_j):=\inf_{u\in S_j}d(x,\,u)$.
In this case, for all $x\in\cx\setminus(6S_j)$ and $y\in S_j$,
$d(x,\,y)\ge \dist(x,\,S_j)>\ez$. From this fact, $\supp(h_j)\subset S_j$
and $\int_{S_j}h_j(y)\,d\mu(y)=0$, it follows that,
for all $x\in\cx\setminus6S_j$,
\begin{eqnarray*}
|T_\ez(h_j)(x)|&=&
\lf|\int_{d(x,\,y)>\ez}\lf[K(x,\,y)-K(x,\,x_{B_j})\r]h_j(y)\,d\mu(y)\r|\\
&\le&\int_\cx|K(x,\,y)-K(x,\,x_j)||h_j(y)|\,d\mu(y).
\end{eqnarray*}

{\rm Case (ii)} $\ez\in(\dist(x,\,S_j)+2r_{S_j},\,\infty)$.
In this case, noticing that $\supp(\fai_j)\subset S_j$ and
$\supp(\oz_j)\subset  2B_j\subset S_j$, we see that, for all $x\in\cx\setminus6S_j$ and
$y\in\supp(h_j)$,
$d(x,\,y)\le\dist(x,\,S_j)+2r_{S_j}<\ez$ and hence $T_\ez(h_j)(x)=0$.

{\rm Case (iii)} $\ez\in[\dist(x,\,S_j),\,\dist(x,\,S_j)+2r_{S_j}]$.
In this case, observe that, for all $x\in\cx\setminus6S_j$, $\dist(x,\,S_j)>
r_{S_j}$ and hence $\ez<3\dist(x,\,S_j)$. Therefore, by $\supp(h_j)\subset S_j$,
$\int_{S_j}h_j(y)\,d\mu(y)=0$ and \eqref{eq1.5},
we conclude that,
for all $x\in\cx\setminus6S_j$,
\begin{eqnarray*}
|T_\ez(h_j)(x)|
&=&\lf|\int_{d(x,\,y)>\ez}K(x,\,y)h_j(y)\,d\mu(y)\r|\\
&\le&\lf|\int_{d(x,\,y)>\dist(x,\,S_j)}K(x,\,y)h_j(y)\,d\mu(y)\r|
+\lf|\int_{\dist(x,\,S_j)<d(x,\,y)\le\ez}\dots\r|\\
&\le&\int_\cx|K(x,\,y)-K(x,\,x_{B_j})||h_j(y)|\,d\mu(y)
+\int_{\ez/3<d(x,\,y)\le\ez}|K(x,\,y)h_j(y)|\,d\mu(y)\\
&\ls&\int_\cx|K(x,\,y)-K(x,\,x_{B_j})||h_j(y)|\,d\mu(y)
+\frac{1}{\lz(x,\,2\ez)}\int_{d(x,\,y)<2\ez}|h_j(y)|\,d\mu(y).
\end{eqnarray*}

Combining the estimates in Cases (i), (ii) and (iii), we
see that, for all $\ez\in(0,\,\fz)$ and  $x\in\cx\setminus6S_j$,
$$|T_\ez(h_j)(x)|\ls
\int_\cx|K(x,\,y)-K(x,\,x_{B_j})||h_j(y)|\,d\mu(y) +\frac{1}{\lz(x,\,2\ez)}
\int_{d(x,\,y)<2\ez}|h_j(y)|\,d\mu(y).$$
Since, for all $\ez\in(0,\,\fz)$ and $x\in\cx\setminus6S_j$,
$$\frac{1}{\lz(x,\,2\ez)}
\int_{d(x,\,y)<2\ez}|h_j(y)|\,d\mu(y) \ls
\frac{1}{\mu(B(x,10\ez))}\int_{d(x,\,y)<2\ez}|h_j(y)|\,d\mu(y),$$
then,  for all $x\in\cx\setminus6S_j$,
\begin{eqnarray}\label{eq2.x5}
{\rm A}_2(x)&\ls&\sum_j\int_\cx|K(x,\,y)-K(x,\,x_{B_j})|
|h_j(y)|\,d\mu(y)\chi_{\cx\setminus6S_j}(x)\\
&\hs&+\sup_{\ez\in(0,\,\infty)}
\sum_j\frac{1}{\mu(B(x,\,10\ez))}\int_{d(x,\,y)<2\ez}|h_j(y)|\,d\mu(y)
\chi_{\cx\setminus6S_j}(x)\n\\
&\ls&\sum_j\int_\cx|K(x,\,y)-K(x,\,x_{B_j})||h_j(y)|\,d\mu(y)
\chi_{\cx\setminus6S_j}(x)\n\\
&\hs&+M_{(5)}\lf(\sum_j\lf|h_j\r|\r)(x)
=:{\rm A}_{2,\,1}(x)+{\rm A}_{2,\,2}(x).\n
\end{eqnarray}

From \eqref{eq1.6}, the fact that $\supp h_j\subset S_j$ and Lemma \ref{lem2.2}, we infer that
\begin{eqnarray*}
&&\mu\lf(\lf\{x\in\cx\setminus
\lf(\bigcup_j6^2B_j\r):\, {\rm A}_{2,\,1}(x)>t\r\}\r)\\
&&\hs\le t^{-1}\sum_j\int_{\cx\setminus6S_j}
\int_\cx|K(x,\,y)-K(x,\,x_{B_j})||h_j(y)|\,d\mu(y)\,d\mu(x)\\
&&\hs\ls t^{-1}\sum_j\int_\cx|h_j(y)|\,d\mu(y)\ls
t^{-1}\sum_j\|h_j\|_\lo\\
&&\hs\ls t^{-1}\sum_j\lf[\int_{\cx}|f(y)|w_j(y)\,d\mu(y)
+\|\fai_j\|_{L^\infty(\mu)}\mu(S_j)\r]\ls t^{-1}\|f\|_\lo.
\end{eqnarray*}

Similarly, by the fact that $M_{(5)}$ is bounded
from $L^1(\mu)$ into $L^{1,\,\infty}(\mu)$, together with
Lemma \ref{lem2.2}, we see  that
$$\mu\lf(\lf\{x\in\cx\setminus
\lf(\bigcup_j6^2B_j\r): \,{\rm A}_{2,\,2}(x)>t\r\}\r)\ls t^{-1}\sum_j
\int_\cx|h_j(x)|\,d\mu(x) \ls t^{-1}\|f\|_\lo.$$

Combining the estimates for ${\rm A}_1(x)$ and ${\rm A}_2(x)$, we obtain the desired estimate
\eqref{eq2.13}. Thus, we prove \eqref{eq2.12}, which completes
the proof of (i) $\Longrightarrow$ (ii).

To prove that ${\rm (ii)}\Longrightarrow{\rm (iii)}$, let $r\in(0,\,1)$. We first claim that
there exists a positive constant $C_1$ such that,
for all $f\in L^\infty_b(\mu)$,
\begin{equation}\label{eq2.16}
\lf\|M_r^\sharp(T_\ast f)\r\|_\lin\le C_1\|f\|_\lin.
\end{equation}

To show this, for any ball $B\subset\cx$ and $r\in(0,\,1)$, let
$$h_{B,\,r}:=m_B\lf(\lf[T_\ast\lf(f\chi_{\cx\setminus2B}\r)\r]^r\r).$$
Observe that, for any ball $B\subset\cx$,
\begin{eqnarray*}
&&\frac{1}{\mu(5B)}\int_B|[T_\ast(f)(x)]^r
-m_{\wz B}([T_\ast(f)]^r)|\,d\mu(x)\\
&&\hs\le\frac{1}{\mu(5B)}\int_B|[T_\ast(f)(x)]^r
-h_{B,\,r}|\,d\mu(x)\\
&&\hs\hs+|h_{B,\,r}-h_{\wz B,\,r}|
+\frac{1}{\mu(\wz B)}\int_{\wz B}\lf|[T_\ast(f)(x)]^r -h_{\wz B,\,r}\r|\,d\mu(x)
\end{eqnarray*}
and, for two doubling balls $B\subset S$,
\begin{eqnarray*}
|m_B([T_\ast(f)]^r)-m_S([T_\ast(f)]^r)|&\le&|m_B([T_\ast(f)]^r)-h_{B,\,r}|\\
&&\hs+|h_{B,\,r}-h_{S,\,r}|+|h_{S,\,r}-m_S([T_\ast(f)]^r)|.
\end{eqnarray*}
We claim that to show \eqref{eq2.16}, it suffices to prove that, for all balls $B\subset\cx$,
\begin{equation}\label{eq2.17}
{\rm D}_1:=\frac{1}{\mu(5B)}\int_B\lf|[T_\ast(f)(x)]^r
-h_{B,\,r}\r|\,d\mu(x)\ls\|f\|_\lin^r
\end{equation}
and, for all balls $B\subset S\subset\cx$ with $S$ being the doubling ball,
\begin{equation}\label{eq2.18}
{\rm D}_2:=|h_{B,\,r}-h_{S,\,r}|\ls[1+\dz(B,\,S)]^r\|f\|_\lin^r.
\end{equation}
Indeed, assuming that \eqref{eq2.17} and \eqref{eq2.18}
are true, then from the $(6, \bz_6)$-doubling property
of $\wz B$ for any ball $B$, \eqref{eq2.18} and Lemma \ref{lem2.1}(iii), it follows that
$$\lf|h_{B,\,r}-h_{\wz B,\,r}\r|\ls\lf[1+\dz(B,\,\wz B)\r]^r\|f\|_\lin^r\ls \|f\|_\lin^r$$
and, from \eqref{eq2.17}, that for any $(6, \bz_6)$-doubling balls $B\subset S$,
\begin{eqnarray*}
&&|m_B([T_\ast(f)]^r)-h_{B,\,r}|+|h_{S,\,r}-m_S([T_\ast(f)]^r)|\\
&&\quad\le \dfrac1{\mu(B)}\int_B\lf|[T_\ast(f)(x)]^r-h_{B,\,r}\r|\,d\mu(x)
+\dfrac1{\mu(S)}\int_S\lf|[T_\ast(f)(x)]^r-h_{S,\,r}\r|\,d\mu(x)\\
&&\quad\ls\|f\|_\lin^r,
\end{eqnarray*}
which further implies \eqref{eq2.16}.

We now show \eqref{eq2.17}.
To this end, write
\begin{eqnarray*}
{\rm D}_1&\le&\frac{1}{\mu(5B)}\int_B
|[T_\ast(f)(x)]^r-[T_\ast(f\chi_{\cx\setminus2B})(x)]^r|\,d\mu(x)\\
&&\hs+\frac{1}{\mu(5B)}\int_B
|[T_\ast(f\chi_{\cx\setminus2B})(x)]^r-h_{B,\,r}|d\mu(x)\\
&\le&\frac{1}{\mu(5B)}\int_B
[T_\ast(f\chi_{2B})(x)]^r\,d\mu(x)\\
&&\hs+\frac{1}{\mu(5B)}\frac1{\mu(B)}\int_B\int_B\sup_{\ez\in(0,\,\infty)}
|T_\ez(f\chi_{\cx\setminus2B})(x)-T_\ez(f\chi_{\cx\setminus2B})(y)|^r\,d\mu(y)\,d\mu(x)\\
&=:&{\rm D}_{1,\,1}+{\rm D}_{1,\,2}.
\end{eqnarray*}

The weak type (1,1) estimate of $T_\ast$ and the
Kolmogorov inequality (see, for example, \cite[p.\,91]{g08}) say that
$${\rm D}_{1,\,1}\ls\frac{[\mu(2B)]^{1-r}}{\mu(5B)}
\|f\chi_{2B}\|_\lo^r\ls\|f\|_\lin^r.$$

To estimate ${\rm D}_{1,\,2}$, we first see that, for all
$x,\,y\in B$ and $u\in\cx\setminus2B$, it holds that $d(x,\,u)>r_B$ and $d(y,\,u)>r_B$.
Moreover, for $\ez\in(r_B,\,\infty)$, $x,\,y\in B$ and $u\in\cx\setminus2B$, if
$d(x,\,u)\le\ez$ and $d(y,\,u)>\ez$, then
$d(y,\,u)\le d(y,\,x)+d(x,\,u)<2r_B+\ez\le3\ez$ and, if
$d(x,\,u)>\ez$ and $d(y,\,u)\le\ez$, then
$d(x,\,u)\le d(x,\,y)+d(y,\,u)<2r_B+\ez\le3\ez$.
This, along with \eqref{eq1.5} and \eqref{eq1.6}, implies that, if $\ez\in(r_B,\,\infty)$, then
\begin{eqnarray*}
&&\lf|T_\ez\lf(f\chi_{\cx\setminus2B}\r)(x)
-T_\ez\lf(f\chi_{\cx\setminus2B}\r)(y)\r|\\
&&\hs\le\lf|\int_{\gfz{d(x,\,u)>\ez}{d(y,\,u)>\ez}}[K(x,u)-K(y,\,u)]f(u)
\chi_{\cx\setminus2B}(u)\,d\mu(u)\r|\\
&&\hs\hs+\lf|\int_{\gfz{d(x,\,u)>\ez}{d(y,\,u)\le\ez}}K(x,\,u)f(u)
\chi_{\cx\setminus2B}(u)\,d\mu(u)\r|
+\lf|\int_{\gfz{d(y,\,u)>\ez}{d(x,\,u)\le\ez}}K(y,\,u)f(u)
\chi_{\cx\setminus2B}(u)\,d\mu(u)\r|\\
&&\hs\le\int_{\gfz{d(x,\,u)>4\ez}{d(y,\,u)>\ez}}
|K(x,\,u)-K(y,\,u)||f(u)|\,d\mu(u)+\int_{\gfz{\ez<d(x,\,u)\le4\ez}{d(y,\,u)>\ez}}
|K(x,\,u)||f(u)|\,d\mu(u)\\
&&\hs\hs+\int_{\gfz{\ez<d(x,\,u)\le4\ez}{d(y,\,u)>\ez}}
|K(y,\,u)||f(u)|\,d\mu(u)+\int_{\ez< d(x,\,u)<3\ez}
|K(x,\,u)||f(u)|\,d\mu(u)\\
&&\hs\hs+\int_{\ez< d(y,\,u)<3\ez}
|K(y,\,u)||f(u)|\,d\mu(u)\\
&&\hs\ls\|f\|_\lin+\int_{\ez<d(x,\,u)\le4\ez}|K(x,\,u)||f(u)|\,d\mu(u)
+\int_{\ez<d(y,\,u)<6\ez}|K(y,\,u)||f(u)|\,d\mu(u)\\
&&\hs\ls\|f\|_\lin
\end{eqnarray*}
and, if $\ez\in(0,\,r_B]$, then
\begin{eqnarray*}
&&\lf|T_\ez\lf(f\chi_{\cx\setminus2B}\r)(x)
-T_\ez\lf(f\chi_{\cx\setminus2B}\r)(y)\r|\\
&&\hs\le\lf|\int_{\cx}[K(x,\,u)-K(y,\,u)]f(u)
\chi_{\cx\setminus2B}(u)\,d\mu(u)\r|\\
&&\hs\le\int_{\cx\setminus2B}
|K(x,\,u)-K(y,\,u)||f(u)|\,d\mu(u)\\
&&\hs\le \int_{\gfz{d(x,\,u)>4r_B}{d(y,\,u)>r_B}}|K(x,\,u)-K(y,\,u)||f(u)|\,d\mu(u)
+\int_{\gfz{r_B<d(x,\,u)\le4r_B}{d(y,\, u)>r_B}}\dots\\
&&\hs\ls\|f\|_\lin+\int_{r_B<d(x,\,u)\le4r_B}|K(x,\,u)||f(u)|\,d\mu(u)\\
&&\hs\hs+\int_{r_B<d(y,\,u)<6r_B}|K(y,\,u)||f(u)|\,d\mu(u)\ls\|f\|_\lin.
\end{eqnarray*}
Therefore,
${\rm D}_{1,\,2}\ls\|f\|_\lin^r.$

Combining the estimates for ${\rm D}_{1,\,1}$ and ${\rm D}_{1,\,2}$, we obtain the desired estimate
\eqref{eq2.17}.

We now turn to the proof of \eqref{eq2.18}.
To this end, denote by $N_{B,\,S}$ the smallest integer $k$ with $k\ge 2$
such that $2S\subset(\frac3{2})^kB$ and $N_{B,\,S}$ simply by
$N$. Clearly,
\begin{align*}
|h_{B,\,r}-h_{S,\,r}|&=\lf|m_B\lf(\lf[T_\ast\lf(f\chi_{\cx\setminus2B}\r)\r]^r\r)
-m_S\lf(\lf[T_\ast\lf(f\chi_{\cx\setminus2S}\r)\r]^r\r)\r|\\
&\le\lf|m_B\lf(\lf[T_\ast\lf(f\chi_{(\frac3{2})^{N}B\setminus2B}\r)\r]^r\r)\r|
+\lf|m_S\lf(\lf[T_\ast\lf(f\chi_{(\frac3{2})^{N}B\setminus2S}\r)\r]^r\r)\r|\\
&\hs\hs+\lf|m_B\lf(\lf[T_\ast\lf(f\chi_{\cx\setminus(\frac3{2})^{N}B}\r)\r]^r\r)
-m_S\lf(\lf[T_\ast\lf(f\chi_{\cx\setminus(\frac3{2})^{N}B}\r)\r]^r\r)\r|\\
&=:{\rm D}_{2,\,1}+{\rm D}_{2,\,2}+{\rm D}_{2,\,3}.
\end{align*}

From \eqref{eq1.5}, \eqref{eq1.3} and \eqref{eq1.4}, we deduce that, for all $\ez\in(0,\fz)$ and $y\in B$,
\begin{eqnarray*}
&&\lf|T_\ez\lf(f\chi_{(\frac3{2})^{N}B\setminus2B}\r)(y)\r|\\
&&\hs\ls\sum_{k=2}^{N-1}\int_{(\frac3{2})^{k+1}B\setminus(\frac3{2})^kB}
\frac{|f(u)|}{\lz(y,\,d(y,\,u))}\,d\mu(u)+\int_{\frac9{4}B\setminus2B}
\frac{|f(u)|}{\lz(y,\,d(y,\,u))}\,d\mu(u)\\
&&\hs\ls\sum_{k=2}^{N-1}\int_{(\frac3{2})^{k+1}B\setminus(\frac3{2})^kB}
\frac{|f(u)|}{\lz(x_B,\,d(x_B,\,u))}\,d\mu(u)+\int_{\frac9{4}B\setminus2B}
\frac{|f(u)|}{\lz(x_B,\,d(x_B,\,u))}\,d\mu(u)\\
&&\hs\ls\|f\|_\lin[1+\dz(B,\,S)],
\end{eqnarray*}
where we used the fact that
$$\int_{\frac9{4}B\setminus2B}
\frac{|f(u)|}{\lz(x_B,\,d(x_B,\,u))}\,d\mu(u)\ls \int_{\frac9{4}B\setminus2B}
\frac1{\lz(x_B,\,2r_B)}\,d\mu(u)\|f\|_\lin\ls\|f\|_\lin.$$
This implies that
${\rm D}_{2,\,1}\ls[1+\dz(B,\,S)]^r\|f\|_\lin^r.$

By the weak type (1, 1) of $T_\ast$, the Kolmogorov inequality, the fact that $(\frac3{2})^{N}B\subset6S$
and the $(6, \bz_6)$-doubling property of $S$, we see that
$${\rm D}_{2,\,2}\ls\frac1{[\mu(S)]^r}
\lf\|f\chi_{(\frac3{2})^{N}B\setminus2S}\r\|_\lo^r\ls\|f\|_\lin^r.$$

From the trivial inequality, $|a|^r-|b|^r\le|a-b|^r$ for all $a,\,b\in\cc$ and
$r\in(0,1)$, and some argument similar to that for ${\rm D}_{1,\,2}$,
we infer that
${\rm D}_{2,\,3}\ls\|f\|_\lin^r.$

Combining the estimates for ${\rm D}_{2,\,i}$ with
$i\in\{1,2,3\}$, we obtain \eqref{eq2.18}.

We now conclude the proof of (ii) $\Longrightarrow$ (iii) by
considering the following two cases for
$\mu(\cx)$.

Case (i)\ $\mu(\cx)=\fz$. We first claim that, for all $r\in(0,\,1)$,
$p\in[1,\,\infty)$ and $f\in L^\infty_b(\mu)$,
\begin{equation}\label{weak-bound}
\|M^\sharp_r(T_\ast f)\|_{L^{p,\,\infty}(\mu)}\ls\|f\|_{L^p(\mu)}.
\end{equation}

Indeed, it is not clear whether the operator $M^\sharp_r\circ T_\ast$ is quasi-linear.
However, we still see that there exists a positive constant $C_2$ such that,
for all $f_1,\,f_2\in L^\infty_b(\mu)$ and $x\in\cx$,
\begin{equation}\label{control}
M^\sharp_r(T_\ast(f_1+f_2))(x)
\le C_2[M_{r,\,(5)}(T_\ast f_1)(x)+M^\sharp_r(T_\ast f_2)(x)].
\end{equation}
Indeed, by $r\in(0, 1)$, we see that for all $f_1,\,f_2\in L^\infty_b(\mu)$, $x\in\cx$ and $B\ni x$,
\begin{eqnarray*}
&&\dfrac1{\mu(5B)}\dint_B\lf||T_\ast(f_1+f_2)(y)|^r-m_{\wz B}(|T_\ast(f_1+f_2)|^r)\r|\,d\mu(y)\\
&&\quad\le \dfrac1{\mu(5B)}\dint_B\lf||T_\ast(f_2)(y)|^r-m_{\wz B}(|T_\ast(f_2)|^r)\r|\,d\mu(y)+m_{\wz B}(|T_\ast(f_1)|^r)\\
&&\quad\quad+ \dfrac1{\mu(5B)}\dint_B|T_\ast(f_1)(y)|^r\,d\mu(y)
\ls \lf[M^\sharp_r(T_\ast f_2)(x)\r]^r+\lf[M_{r,\,(5)}(T_\ast f_1)(x)\r]^r
\end{eqnarray*}
and, for any $(6, \bz_6)$-doubling balls $B\subset S$ with $B\ni x$,
\begin{eqnarray*}
&&\lf|m_B(|T_\ast(f_1+f_2)|^r)-m_S(|T_\ast(f_1+f_2)|^r)\r|\\
&&\quad\le \lf|m_B(|T_\ast(f_2)|^r)-m_S(|T_\ast(f_2)|^r)\r|+m_B(|T_\ast(f_1)|^r)+m_S(|T_\ast(f_1)|^r)\\
&&\quad\ls [1+\dz(B, S)]\lf[M^\sharp_r(T_\ast f_2)(x)\r]^r+\lf[M_{r,\,(5)}(T_\ast f_1)(x)\r]^r.
\end{eqnarray*}
Combining these two estimates yields \eqref{control}.

By the boundedness, from $L^1(\mu)$
into $L^{1,\,\infty}(\mu)$, of $T_\ast$ and $M_{(5)}$, we conclude that, for all $r\in(0,\,1)$,
$M_{r,\,(5)}\circ T_\ast$ is also bounded from $L^1(\mu)$
into $L^{1,\,\infty}(\mu)$; see \cite[Lemma 3.2]{ntv2} for the details.
For all $f\in L^\infty_b(\mu)$ and $t\in(0,\,\infty)$, we split
$f$ into $f_1$ and $f_2$ with $f_1:=f\chi_{\{y\in\cx:\ |f(y)|>t\}}$
and $f_2:=f\chi_{\{y\in\cx:\ |f(y)|\le t\}}$.
It is easy to see that
$$\|f_1\|_{L^1(\mu)}\le t^{1-p}\|f\|^p_{L^p(\mu)}\
\mbox{and}\
\|f_2\|_{L^\infty(\mu)}\le t.$$
By this fact, \eqref{control}, \eqref{eq2.16}, the boundedness of  $M_{r,\,(5)}\circ T_\ast$
from $L^1(\mu)$ into $L^{1,\,\infty}(\mu)$ with $r\in(0,\,1)$, we deduce that, for all $t\in(0,\,\infty)$,
\begin{eqnarray*}
\mu(\{x\in\cx:\ M^\sharp_r(T_\ast f)(x)>C_2(1+C_1)t\})
&\le&\mu(\{x\in\cx:\ M_{r,\,(5)}(T_\ast f_1)(x)>t\})\\
&\ls& t^{-1}\|f_1\|_{L^1(\mu)}\ls t^{-p}\|f\|^p_{L^p(\mu)},
\end{eqnarray*}
which implies the desired result \eqref{weak-bound}.

Now we show that,  for all $f\in L^\infty_b(\mu)$, $\min\{1,\ N_r(T_\ast f)\}\in\lp$, where
$N_r(g):=[N(|g|^r)]^{1/r}$ for all $g\in L_{\loc}^r(\mu)$.
Indeed,  for all $f\in L^\infty_b(\mu)$, we see that $f\in\lo$.
Moreover, by the definitions of $N_r$ and $M_{r,(5)}$ with $r\in(0,1)$,
we easily conclude that, for $\mu$-almost every $x\in\cx$,
$N_r(T_\ast f)(x)\ls M_{r,\,(5)}(T_\ast f)(x)$.
From the assumption of Theorem \ref{thm1.2}(ii) and the fact that
$M_{r,(5)}$ is bounded from $\wlo$ into $\wlo$ (see \cite[Lemma 2.3]{hlyy}), we deduce that
$$\|N_r(T_\ast f)\|_\wlo\ls\|M_{r,\,(5)}(T_\ast f)\|_\wlo\ls
\|T_\ast f\|_\wlo\ls\|f\|_\lo<\infty,$$
which implies that $N_r(T_\ast f)$ lies in $\wlo$ for all $r\in(0,\,1)$.
Therefore,  for all $p\in(1,\,\fz)$,
\begin{align*}
&\int_\cx\lf[\min\{1,\ N_r(T_\ast f)\}(x)\r]^p\,d\mu(x)\\
&\hs=p\int_0^2t^{p-1}
\mu\lf(\{x\in\cx:\ \min\{1,\ N_r(T_\ast f)\}(x)>t\}\r)\,dt+p\int_2^\fz\dots\\
&\hs\ls p\int_0^2t^{p-1}\mu\lf(\{x\in\cx:\ \min\{1,\ N_r(T_\ast f)\}(x)>t\}\r)\,dt\\
&\hs\ls\|N_r(T_\ast f)\|_\wlo\int_0^2t^{p-2}\,dt
\ls\|N_r(T_\ast f)\|_\wlo.
\end{align*}
Thus, $\min\{1,\ N_r(T_\ast f)\}\in\lp$.

From this, \eqref{eq2.9x}, Lemma \ref{lem2.1x} and \eqref{weak-bound},
it follows that, for all $p\in(1,\,\fz)$,
\begin{eqnarray*}
\|T_\ast(f)\|_{L^{p,\,\infty}(\mu)}
&&\le\|N_r(T_\ast f)\|_{L^{p,\,\infty}(\mu)}
=\|N(|T_\ast f|^r)\|^{\frac1r}_{L^{\frac pr,\,\infty}(\mu)}\\
&&\ls \|M^\sharp(|T_\ast f|^r)\|^{\frac1r}_{L^{\frac pr,\,\infty}(\mu)}
\sim \|M^\sharp_r(T_\ast f)\|_{L^{p,\,\infty}(\mu)}\ls\|f\|_{L^p(\mu)},
\end{eqnarray*}
which, along with the Marcinkiewicz interpolation theorem, implies that, for all $p\in(1,\,\fz)$,
$$\|T_\ast f\|_\lp\ls\|f\|_\lp.$$
Therefore, (iii) holds in this case.

Case (ii) $\mu(\cx)<\fz$.
In this case, for all $r\in(0,\,1)$, $f\in L^\infty_b(\mu)$ and $x\in\cx$, we see that
\begin{align*}
|T_\ast(f)(x)|&\le\lf[N\lf(|T_\ast(f)|^r\r)
(x)\r]^\frac{1}{r}\\
&\ls\lf\{N\lf[|T_\ast(f)|^r-\frac{1}{\mu(\cx)}
\int_\cx|T_\ast(f)|^r\,d\mu\r](x)\r\}^\frac{1}{r}\\
&\hs+\lf\{\frac{1}{\mu(\cx)}
\int_\cx|T_\ast(f)(x)|^r\,d\mu(x)\r\}^\frac{1}{r}
=:{\rm U}(x)+\lf\{\frac{1}{\mu(\cx)}
\int_\cx|T_\ast(f)(x)|^r\,d\mu(x)\r\}^\frac{1}{r}.
\end{align*}

The same argument as in the case that $\mu(\cx)=\fz$
gives us the desired estimate for ${\rm U}(x)$.

Recall that $T_\ast$ is bounded from $\lo$ into $\wlo$.
From this, it follows that, for all $r\in(0,\,1)$,
\begin{align}\label{eq2.20}
&\int_\cx|T_\ast(f)(x)|^r\,d\mu(x)\\
&\hs=r\int_0^{\|f\|_\lo[\mu(\cx)]^{-1}}
t^{r-1}\mu\lf(\{x\in\cx:\ |T_\ast(f)(x)|>t\}\r)\,dt+r\int_{\|f\|_\lo[\mu(\cx)]^{-1}}^\fz\dots\n\\
&\hs\ls\int_0^{\|f\|_\lo[\mu(\cx)]^{-1}}
t^{r-1}\mu(\cx)\,dt+\int_{\|f\|_\lo[\mu(\cx)]^{-1}}^\fz
t^{r-2}\|f\|_\lo\,dt\n\\
&\hs\ls[\mu(\cx)]^{1-r}\|f\|_\lo^r\ls\|f\|_\lp^r.\n
\end{align}

Combining the estimate for ${\rm U}(x)$ and \eqref{eq2.20},
we see that $T_\ast$ is bounded on $\lp$ for all
$p\in(1,\,\fz)$, which completes the
proof of (ii) $\Longrightarrow$ (iii).

The proof of (iii) $\Longrightarrow$ (i) is obvious. This finishes the proof of Theorem \ref{thm1.2}.
\end{proof}

\section{Proof of Theorem \ref{thm1.1}}\label{sec3}

\hskip\parindent
The purpose of this section is to prove Theorem \ref{thm1.1}. To this end,
we first establish a Cotlar type inequality. Indeed, such an inequality
was first obtained by Grafakos \cite{gr} in the classical Euclidean space $\rd$
with the $d$-dimensional Lebesgue measure.
Later, this Coltar type inequality was generalized to
$(\rd,\,|\cdot|,\,\mu)$ with the measure $\mu$ satisfying \eqref{eq1.2} in \cite[Theorem 3.1]{hmy06}.
We point out that  Bui and Doung in \cite{ad} obtained another Coltar type inequality
for the Calder\'on-Zygmund operator with the kernel
satisfying \eqref{eq1.5} and \eqref{eq1.11} on non-homogeneous
metric measure spaces. However, their Coltar type inequality is not valid
under the present assumptions, since
the regularity condition \eqref{eq1.11} of the kernel $K$
is stronger than \eqref{eq1.6}. Therefore, we establish a Coltar
type inequality different from theirs as follows.

\begin{thm}\label{thm3.1}
Let $K$ be a $\mu$-locally integrable function mapping from
$(\cx\times\cx)\setminus\triangle$ to $\cc$,
which satisfies \eqref{eq1.5} and \eqref{eq1.6}, and
$T$ and $T_\ast$ as in \eqref{eq1.7} and \eqref{eq1.9}, respectively.
Suppose that $T$ is bounded on $\lt$. Then there exists a positive constant $C_3$
such that, for all $f\in L^2(\mu)\cap L^\infty(\mu)$ and $\mu$-almost every $x\in\supp\cx$,
\begin{equation}\label{coltar}
T_\ast(f)(x)\le C_3\lf[M_{(5)}(Tf)(x)+\|f\|_\lin\r].
\end{equation}
\end{thm}

\begin{proof}
We show this theorem by borrowing some ideas from \cite[Theorem 6.6]{ad}.
For all $\ez\in(0,\,\infty)$ and $x\in\cx$, let $B_x$ be the biggest
$(6,\bz_6)$-doubling ball with center $x$ of the form $6^{-k}\ez$ with $k\in\nn$.
Let $B_x:=B(x,6^{-k_0}\ez)$.
Split $f=f_1+f_2$, where, $f_1:=f\chi_{3B_x}$ and
$f_2:=f\chi_{\cx\setminus3B_x}$.
Notice that
$$\{y\in\cx:\ d(x,y)>\ez\}\cap3B(x,6^{-k_0}\ez)=\emptyset$$
and hence, for each $\ez\in(0, \fz)$, $T_\ez(f_1)(x)=0$.

Now we estimate $T_\ez(f_2)(x)$.
To this end, we write that, for all $\ez\in(0,\,\infty)$ and $x\in\cx$,
$$|T_\ez(f_2)(x)|\le|T(f_2)(x)|
+\lf|\int_{d(x,\,y)\le\ez} K(x,\,y)f_2(y)\,d\mu(y)\r|=:{\rm I}_1+{\rm I}_2.$$

For the term ${\rm I}_1$, we further write that, for all $u\in B_x$,
$$|T(f_2)(x)|\le|T(f_2)(x)-T(f_2)(u)|+|T(f)(u)|+|T(f_1)(u)|.$$
Applying the H\"{o}rmander condition \eqref{eq1.6},
we conclude that, for all $x\in\cx$ and $u\in B_x$,
$$|T(f_2)(x)-T(f_2)(u)|\le\int_{\cx\setminus3B_x}
|K(x,\,y)-K(u,\,y)||f(y)|\,d\mu(y)\ls\|f\|_\lin.$$
This implies that, for all $u\in B_x$,
\begin{equation}\label{eq3.1}
|T(f_2)(x)|\ls\|f\|_\lin+|T(f_1)(u)|+|T(f)(u)|.
\end{equation}
Integrating the inequality \eqref{eq3.1} with respect to the
variable $u$ over the ball $B_x$, from the H\"older inequality and
the boundedness of $T$ on $L^2(\mu)$, we deduce that, for all $x\in\cx$,
\begin{eqnarray*}
|T(f_2)(x)|&\ls&\|f\|_\lin+\frac{1}{\mu(B_x)}\int_{B_x}|T(f_1)(u)|\,d\mu(u)
+\frac{1}{\mu(B_x)}\int_{B_x}|T(f)(u)|\,d\mu(u)\\
&\ls&\|f\|_\lin+\frac{1}{\mu(B_x)}\|T(f_1)\|_{L^2(\mu)}[\mu(B_x)]^{1/2}
+\frac{1}{\mu(5B_x)}\int_{B_x}|T(f)(u)|\,d\mu(u)\\
&\ls&\|f\|_\lin+\frac{1}{\mu(B_x)}\|T(f_1)\|_{L^2(\mu)}[\mu(B_x)]^{1/2}
+M_{(5)}(Tf)(x)\\
&\ls&\|f\|_\lin+M_{(5)}(T(f))(x).
\end{eqnarray*}

To estimate the term ${\rm I}_2$, by an argument similar to that used for the estimate of ${\rm I}_2$
in \cite[p.\,28]{ad}, we conclude that
${\rm I}_2\ls\|f\|_\lin.$

Combining the estimates for ${\rm I}_1$ and ${\rm I}_2$, we obtain \eqref{coltar},
which completes the proof of Theorem \ref{thm3.1}.
\end{proof}

\begin{proof}[Proof of Theorem \ref{thm1.1}]
By Theorem \ref{thm1.2}, it suffices to show that $T_\ast$ is
bounded from $\lo$ into $\wlo$. Fix any fixed $f\in\lo$ and
$t\in(0,\,\fz)$ ($t>\gz_0\|f\|_\lp/\mu(\cx)$ when $\mu(\cx)<\fz$, since the case
that $t\le\gz_0\|f\|_\lo/\mu(\cx)$  is trivial).
Applying Lemma \ref{lem2.2} in the case that $p=1$, with the same
notation as in the proof of Theorem \ref{thm1.2}, we write that
$f=g+h$. We have proven that there exists a positive constant $C_4$
such that, for almost every $x\in\cx$, $|g(x)|\le C_4t$.
Thus, from Theorem \ref{thm3.1} and $\lt$ boundedness
of $M_{(5)}\circ T$, we deduce that
\begin{eqnarray*}
\mu\lf(\lf\{x\in\cx:\ |T_\ast(g)(x)|>(C_4+1)C_3t\r\}\r)
&&\le\mu\lf(\lf\{x\in\cx:\ |M_{(5)}(Tg)(x)|>t\r\}\r)\\
&&\le t^{-2}\int_\cx|M_{(5)}(Tg)(x)|^2\,d\mu(x)\\
&&\ls t^{-2}\|g\|_\lt^2\ls t^{-1}\|f\|_\lo,
\end{eqnarray*}
where $C_3\in(0, \fz)$ is as in Theorem \ref{thm3.1}.
As in the proof of (i) $\Longrightarrow$ (ii) of Theorem
\ref{thm1.2}, we see that the proof of Theorem
\ref{thm1.1} is reduced to proving that, for all $f\in L^1(\mu)$ and $t\in(0,\,\infty)$,
\begin{equation}\label{eq3.2}
\mu\lf(\lf\{x\in\cx\setminus
\lf(\bigcup_j6^2B_j\r):\ |T_\ast(h)(x)|>t\r\}\r)\ls\frac{1}{t}\|f\|_\lo.
\end{equation}

By the estimates \eqref{eq2.x3}, \eqref{eq2.x4} and \eqref{eq2.x5}
in the proof of (i) $\Longrightarrow$ (ii) of Theorem \ref{thm1.2}, we see that,
for all $x\in\cx\setminus(\cup_j6^2B_j)$,
\begin{align*}
T_\ast(h)(x)&\ls\sum_{j}T_\ast(\varphi_j)(x)\chi_{6S_j}(x)
+\sum_{j}\frac{\chi_{6S_j\setminus(6^2B_j)}(x)}
{\lz(x,\,d(x,\,x_{B_j}))}\int_{B_j}|f(y)|\,d\mu(y)\\
&\hs+\sum_{j}\chi_{\cx\setminus6S_j}(x)
\int_\cx|K(x,\,y)-K(x,\,x_{B_j})||h_j(y)|\,d\mu(y)\\
&\hs+M_{(5)}\lf(\sum_{j}|h_j|\r)(x)
=:{\rm E}(x)+{\rm F}(x)+{\rm G}(x)+{\rm H}(x).
\end{align*}

The estimates for ${\rm F}(x),\ {\rm G}(x)$ and ${\rm H}(x)$ are the same as in the
proof of Theorem \ref{thm1.2}.

We only need to estimate ${\rm E}(x)$.
Applying Theorem \ref{thm3.1}, we conclude that
\begin{eqnarray*}
&&\mu\lf(\lf\{x\in\cx\setminus\lf(\bigcup_j6^2B_j\r):\ {\rm E}(x)>t\r\}\r)\\
&&\hs\le t^{-1}\sum_j\int_{6S_j}|T_\ast(\varphi_j)(x)|\,d\mu(x)\\
&&\hs\ls t^{-1}\lf\{\sum_j\lf[\int_{6S_j}M_{(5)}(T\varphi_j)(x)
\,d\mu(x)+\|\varphi_j\|_\lin\ \mu(6S_j)\r]\r\}\\
&&\hs\ls t^{-1}\lf\{\sum_j\lf[\|M_{(5)}(T\varphi_j)\|_\lt
[\mu(6S_j)]^\frac{1}{2}+\|\varphi_j\|_\lin\ \mu(6S_j)\r]\r\}\ls t^{-1}\|f\|_\lo,
\end{eqnarray*}
which implies \eqref{eq3.2} and hence completes the proof of Theorem \ref{thm1.1}.
\end{proof}

\medskip

{\bf Acknowledgements.} The authors sincerely wish to express their
deeply thanks to the referee for her/his very carefully reading and
also her/his so many careful, valuable and suggestive remarks which essentially
improve the presentation of this article. The authors would also like to thank Dr. Dongyong Yang
for some helpful discussions on the subject of this paper.

\bigskip

\noindent Suile Liu and Dachun Yang (Corresponding author)

\medskip

\noindent School of Mathematical Sciences,
 Beijing Normal University, Laboratory of Mathematics and Complex systems,
Ministry of Education, Beijing 100875, People's Republic of China

\medskip

\noindent{\it E-mail addresses}: \texttt{slliu@mail.bnu.edu.cn} (S. Liu)

                                \hspace{2.58cm}\texttt{dcyang@bnu.edu.cn} (D. Yang)

\medskip

\noindent Yan Meng

\medskip

\noindent School of Information,
 Renmin University of China, Beijing 100872, P. R. China

\medskip

\noindent{\it E-mail address}: \texttt{mengyan@ruc.edu.cn}
\end{document}